\documentclass[a4,12pt]{amsart}       

\usepackage[utf8]{inputenc}
\usepackage[T1]{fontenc}
\usepackage{yfonts}
\usepackage[english]{babel}
\usepackage{lipsum}
\usepackage{amsmath}
\usepackage{amsthm}
\usepackage{amssymb}

\usepackage[shortlabels]{enumitem}
\usepackage{graphicx}
\usepackage{mathtools}
\usepackage{hyperref}
\usepackage{amsfonts}
\usepackage{latexsym}
\usepackage{amscd}
\usepackage[dvipsnames]{xcolor}
\usepackage{tikz-cd}

\usepackage[all]{xy}

\usepackage{stmaryrd}

\DeclareMathOperator{\Imm}{Im}

\DeclareMathOperator{\Path}{Path}

\def\Path{\text{Path}}

\def\dualita#1#2{\mathrel{
                 \mathop{\vcenter{
                 \offinterlineskip
                 \hbox to 1.2truecm{$\mapsto$}
                 \hbox to 1.2truecm{$\mapsfrom$}}}%
                 }}
\usepackage{aliascnt}

\newaliascnt{lemma}{theorem}
\newtheorem{lemma}[lemma]{Lemma}
\aliascntresetthe{lemma}

\newaliascnt{proposition}{theorem}
\newtheorem{proposition}[proposition]{Proposition}
\aliascntresetthe{proposition}

\newaliascnt{corollary}{theorem}

\aliascntresetthe{corollary}

\newaliascnt{definition}{theorem}

\aliascntresetthe{definition}

\newaliascnt{remark}{theorem}
\newtheorem{remark}[remark]{Remark}
\aliascntresetthe{remark}
\newaliascnt{example}{theorem}
\newtheorem{example}[example]{Example}
\aliascntresetthe{example}

\usepackage{cleveref}
\crefname{theorem}{theorem}{theorems}
\Crefname{theorem}{Theorem}{Theorems}

\crefname{lemma}{lemma}{lemmas}
\Crefname{lemma}{Lemma}{Lemmas}

\crefname{proposition}{proposition}{propositions}
\Crefname{proposition}{Proposition}{Propositions}

\crefname{corollary}{corollary}{corollaries}
\Crefname{corollary}{Corollary}{Corollaries}

\crefname{definition}{definition}{definitions}
\Crefname{definition}{Definition}{Definitions}

\crefname{remark}{remark}{remarks}
\Crefname{remark}{Remark}{Remarks}

\crefname{example}{example}{examples}
\Crefname{example}{Example}{Examples}

\begin{document}
\title[divisible non-injective]{On the existence of divisible non-injective modules over Leavitt path algebras}
\author{Mara Barban}
\address{Dipartimento di Matematica ``Tullio Levi-Civita'', Universit\`{a} degli Studi di Padova, I-35121, Padova, Italy, Orcid https://orcid.org/0009-0006-3327-5504}
\email{mara.barban@phd.unipd.it}
\author{Alberto Tonolo}
\address{Dipartimento di Matematica ``Tullio Levi-Civita'', Universit\`{a} degli Studi di Padova, I-35121, Padova, Italy, Orcid https://orcid.org/0000-0002-9844-3998}
\email{alberto.tonolo@unipd.it}

\begin{abstract}
Let $K$ be any field and $E$ any directed graph.  
We investigate the existence of divisible non-injective left modules over the Leavitt path algebra $L_K(E)$ for a broad class of graphs. Our main tool is a general criterion for divisibility that is particularly well suited to our setting.
\end{abstract}
\maketitle

\section*{Introduction}

The algebras now known as \emph{Leavitt path algebras} were introduced in 2004 by Ara, Moreno, and Pardo \cite{AMP07}, and, almost simultaneously and through a different approach, by Abrams and Aranda Pino \cite{AA05}. Over the last two decades, they have attracted considerable interest not only among ring theorists, but also among researchers working in $C^*$-algebras, group theory, and symbolic dynamics. Their construction belongs to the long-standing tradition of associating algebraic objects with suitable combinatorial data: in this case, a directed graph. We refer the reader to the monograph \cite{AAM} for the basic definitions, the principal structural results, and a comprehensive introduction to the subject. The importance acquired by the theory is also reflected in the Mathematics Subject Classification, where Leavitt path algebras are assigned the specific code 16S88.

A particularly active direction of research concerns the module theory of Leavitt path algebras. Initiated in \cite{ArBr10}, this line of investigation has subsequently developed through the work of Chen, Ara, Anh, Nam, Rangaswamy, Abrams, Mantese, and the second author, among others (e.g., see \cite{Ch12, AR14, AMT15, AMT19, AMT21, AMT24, AN21, AMT26}). A substantial part of the literature has focused on the classification and structure of simple modules, as well as on the description of their injective envelopes. These questions reveal the extent to which the combinatorial properties of the underlying graph influence the homological and module-theoretic behaviour of the associated algebra.

The aim of the present work is to compare two natural classes of left modules over Leavitt path algebras: injective left modules and divisible left modules. In general, divisibility is a weaker condition than injectivity, and the two notions need not agree. Their comparison is especially meaningful in the present setting because Leavitt path algebras are hereditary rings \cite{ArG12}. This homological property places them in a particularly delicate position for studying the relationship between divisibility and injectivity. Indeed, for an associative ring without zero divisors, the assertion that divisible modules coincide with injective modules is equivalent to the left hereditariness of the ring.

\section{The problem}

Let $R$ be an arbitrary associative ring.
A left $R$-module $M$ is \emph{injective} \cite[Baer's criterion 3.7]{Lam0} (resp. \emph{divisible} \cite[Proposition 3.17]{Lam0})
if for any left ideal (resp. left principal ideal) $J$ of $R$, and any homomorphism $\varphi:J\to M$ of left $R$-modules there exists a homomorphism of left $R$-modules $\psi:R\to M$ such that the following diagram commutes:
\[
\xymatrix{J\ar[d]_\varphi\ar@{^(->}[r]& R\ar[dl]^\psi\\
M}
\]
Clearly, every injective module is divisible. The converse is immediate when $R$ is a \emph{principal left ideal ring}; however, it fails in general. For example, if $R$ is a domain (not necessarily commutative), every divisible left $R$-module is injective if and only if $R$ is left hereditary \cite[Corollary 3.23]{Lam0}.

In this paper, we want to compare the two notions of injective and divisible left module over Leavitt path algebras.

If $F$ is an acyclic finite graph, then, trivially, no cycle has an exit, i.e., $F$ satisfies Condition (NE). 
By \cite[Theorem 4.2.17]{AAM} $L_K(F)$ is left Noetherian, and hence every left ideal in $L_K(F)$ is finitely generated. Since Leavitt path algebras are B\'ezout rings \cite{AMT18}, we find that every left ideal in $L_K(F)$ is principal. Thus, finite acyclic graphs cannot provide examples of divisible non injective modules.

Nevertheless, such examples may occur among infinite acyclic graphs.

    Consider the infinite Kronecker graph, i.e., the graph $\mathfrak K_\infty$ with two vertices $u$ and $w$, and infinitely many edges $e_i$, $i\geq 1$, with source $u$ and range $w$
\begin{center}
\begin{tikzpicture}[
  >=Stealth,
  line width=0.9pt,
  vertex/.style={inner sep=1.2pt,outer sep=1pt},
  bigloopabove/.style={
    loop above,
    out=130,
    in=50,
    min distance=15mm,
    looseness=4
  },
  edgelabel/.style={
    midway,
    fill=white,
    inner sep=1.2pt,
    font=\footnotesize
  }
]

\node[vertex] (u) at (0,0) {$u$};
\node[vertex] (w) at (4,0) {$w$};

\draw[->] (u) to[bend left=48]
  node[edgelabel,pos=0.50,sloped,inner xsep=1.9pt] {$e_1$} (w);

\draw[->] (u) to[bend left=22]
  node[edgelabel,pos=0.50,sloped,inner xsep=1.9pt] {$e_2$} (w);

\node[fill=white,inner sep=-1.1pt] at (2,0.10) {$\vdots$};

\draw[->] (u) to[bend right=22]
  node[edgelabel,pos=0.50,sloped,inner xsep=1.9pt] {$e_n$} (w);

\node[fill=white,inner sep=-1.1pt] at (2,-0.85) {$\vdots$};

\end{tikzpicture}
\end{center}
By \cite[Theorem 3.4.1]{AAM}, $L_K(\mathfrak K_\infty)$ is von Neumann, and 
hence by \cite[Proposition 3.18]{Lam0} every left $R$-module is divisible: in particular the simple (see \cite[Theorem 3.5]{Ch12}) left ideal $L_K(\mathfrak K_\infty)w$ is divisible. Nevertheless, it is not injective: its injective envelope is the left $L_K(\mathfrak K_\infty)$-module
\[K[[\text{Path}(\mathfrak K_\infty)w]]=\{k_0w+k_1e_1+k_2e_2+\cdots+k_ne_n+\cdots: k_i\in K\}
\]
of the formal series of paths in $\mathfrak K_\infty$ ending in $w$
(see \cite[Corollary 4.6]{AMT26}).

The aim of this paper is to construct a large family of finite graphs admitting divisible non injective left modules over the associated Leavitt path algebra. Notice that, as shown above, these finite graphs necessarily have to contain cycles.

\section{A localisation and an ``ad hoc'' divisibility criterion}

Let $K$ be a field. The set $\mathfrak M$ of polynomials $p(x)\in K[x]$ such that $p(0)\not=0$ is a multiplicative subset of $K[x]$. We can therefore consider the localisation of $K[x]$ at $\mathfrak M$:
\[
K[x]_{\mathfrak M}=\left\{\dfrac{f(x)}{g(x)}\mid f(x), g(x)\in K[x], g(0)\not=0\right\}/\approx
\]
where given $f_1(x), f_2(x)\in K[x]$, and $g_1(x), g_2(x)\in \mathfrak M$
\[
\frac{f_1(x)}{g_1(x)}\approx \frac{f_2(x)}{g_2(x)}\text{ if and only if }f_1(x)g_2(x)=f_2(x)g_1(x) \text{ in }K[x].
\]
Since any polynomial $p(x)$ with $p(0)\not=0$ is invertible in the $K$-algebra of the formal series $K[[x]]$, we can identify $K[x]_{\mathfrak M}$ with a subalgebra of $K[[x]]$. After this identification, it is well known that $K[x]_{\mathfrak M}$ is properly contained in $K[[x]]$: e.g., the generating function of Catalan numbers 
\[
A(x)=\sum_{n=0}^\infty \frac 1{n+1}\binom{2n}nx^n
\]
satisfies the functional equation $A(x)=1+x A^2(x)$ and, therefore, does not belong to $K[x]_{\mathfrak M}$.
Therefore, we have the following strict inclusions of $K$-algebras:
\[
K[x]\subsetneq K[x]_{\mathfrak M}\subsetneq K[[x]].
\]

Consider now any associative ring $R$ with $1\not=0$.
Let us prove the following ``ad hoc'' divisibility criterion for left $R$-modules.

\begin{lemma}\label{Steps}
    Let $R$ be an associative ring and $M$ be a left $R$-module. Then $M$ is divisible if and only if there exists a left ideal $I\leq R$ such that:
    \begin{enumerate}
        \item For every left $R$-ideal $J\leq I$, every left $R$-homomorphism $J\to M$ extends to a left $R$-homomorphism $I\to M$.
        \item For every $r\in I$, every left $R$-homomorphism $Rr\to M$ extends to a left $R$-homomorphism $R\to M$.
        \item For every $r\in R\setminus I$, every left $R$-homomorphism $Rr+I\to M$ extends to a left $R$-homomorphism $R\to M$.
    \end{enumerate}
\end{lemma}
\begin{proof}
Clearly, if $M$ is divisible, then $I=0$ satisfies the three properties. Conversely, assume there exists $I$ satisfying properties (1-3).
    Let $s\in R$ and $\varphi:Rs\to M$ be a left $R$-homomorphism. If $s\in I$ we conclude by (2). Otherwise, let $s\notin I$. By assumption (1) $\varphi_{\mid Rs\cap I}:Rs\cap I\to M$ extends to $\psi:I\to M$. Setting $\hat\psi(i+rs):=\psi(i)+\varphi(rs)$ for each $i\in I$, $r\in R$, we define a map $\hat\psi: I+Rs\to M$. It is well defined: if $i_1+r_1s=i_2+r_2s$, then
    $i_1-i_2=r_2s-r_1s\in I\cap Rs$ and 
    \[
    \psi(i_1)-\psi(i_2)=\psi(i_1-i_2)=\varphi_{\mid Rs\cap I}(i_1-i_2)=\varphi(r_2s-r_1s)=\varphi(r_2s)-\varphi(r_1s),
    \]
 hence $\psi(i_1)+\varphi(r_1s)=\psi(i_2)+\varphi(r_2s)$. Since both $\psi$ and $\varphi$ are left $R$-homomorphisms, $\hat\psi:I+Rs\to M $ is clearly $R$-linear and extends $\varphi$. Finally, by (3) $\hat\psi$ extends to a left $R$-homomorphism $R\to M$.
\end{proof}
\section{Finite graphs with a source loop}

A loop in a graph is an edge $c$ satisfying $s(c)=r(c)$; it is a \emph{source loop} if moreover $r^{-1}(s(c))=c$, i.e., $c$ is the only edge with range $r(c)$.


Let $E$ be a finite graph admitting a source loop $c$ and a sink $w$ with $v:=s(c)\geq w$: the latter means that there exists a path in $E$ connecting the surce $v$ of the loop $c$ and the sink $w$. For each $i\geq 0$, $c^i$ represents the path in $E$ obtained traversing $i$-times the loop $c$; in particular $c^0=s(c)=v$. Let
\[
s^{-1}(v)=\{c, d_1, ..., d_n\}.
\]
The edges $d_1, ..., d_n$ are the \emph{exits} \cite[Definition 2.2.2]{AAM} of the loop $c$.
Consider the set
\[
H:=E^0\setminus \{v\}.
\]
It is easy to check that $H$ is hereditary and saturated. We denote by $I(H)$ the two sided ideal generated by $H$.
\begin{remark}\label{hedgerem}
Both $I(H)$ and $L_K(E)/I(H)$ are Leavitt path algebras:
\begin{itemize}
    \item $I(H)$, when viewed as a $K$-algebra, is isomorphic to the Leavitt path algebra $L_K({}_HE)$, where  ${}_HE$ is the hedgehog graph (see \cite[Definition 2.5.16, Theorem 2.5.19]{AAM}) given by
\begin{align*}
    ({}_HE)^0&=H\sqcup F_E(H)\\
    &F_E(H):=\{c^id_j\mid i\geq 0, 1\leq j\leq n\}\\
    ({}_HE)^1&=\{e\in E^1\mid s(e)\in H\}\sqcup \overline{F_E(H)}\\
    &\overline{F_E(H)}:=\{\overline{c^id_j}\mid i\geq 0, 1\leq j\leq n\}.
\end{align*}
The source and range functions $s', r'$ are defined on  $({}_HE)^1$ setting
\begin{align*}
&s'(e)=s(e), r'(e)=r(e) \quad \forall\ e\in E^1, s(e)\in H\\
& s'(\overline{c^id_j})=c^id_j, r'(\overline{c^id_j})=r(d_j)\quad \forall\ i\geq 0, 1\leq j\leq n.
\end{align*}
Throughout the isomorphism, the vertices in $H\subseteq {}_HE^0$ correspond to the vertices in $H\subseteq I(H)$, the vertices $c^id_j$ in ${}_HE^0$ correspond to the idempotents $c^id_jd_j^*(c^*)^i$ in $I(H)$, the edges in $\{e\in E^1\mid s(e)\in H\}\subseteq ({}_HE)^1$ correspond to the edges $\{e\in E^1\mid s(e)\in H\}\subseteq E^1$, and the edges $\overline{c^id_j}$ correspond to the paths $c^id_j\in I(H)$. Observe that ${}_HE$ has infinitely many vertices and edges. Therefore $I(H)\cong L_K({}_H E)$ is a $K$-algebra without identity. An abelian group $M$ is a left $I(H)$-modules if there is a (standard) left module action of $I(H)$ on $M$, but with the
added proviso that $M$ be unitary, i.e., that $I(H)M= M$.

\item Invoking \cite[Corollary 2.4.13]{AAM} we find that the quotient $L_K(E)/I(H)$ is isomorphic to the Leavitt path algebra $L_K(E/H)$, where $E/H$ is the graph given by:
\begin{align*}
(E/H)^0&=E^0\setminus H=\{v\}\\
(E/H)^1&=\{c\}.
\end{align*}
The source and range functions are the restrictions of those of $E$.
\end{itemize}
Clearly $L_K(E/H)$ is isomorphic to the $K$-algebra  $K[x,x^{-1}]$ of Laurent polynomials.
\end{remark}

\begin{example}
    Consider the graph $E$
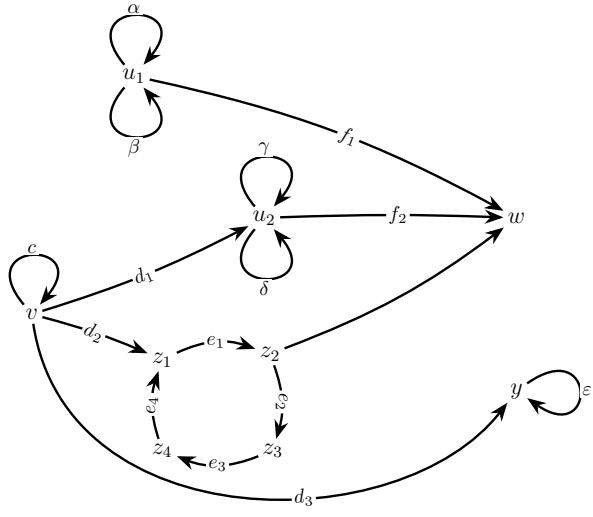
    \begin{center}
        \begin{tikzpicture}[
  scale=0.75,
  transform shape,
 >=Stealth,
  line width=0.9pt,
  vertex/.style={inner sep=1.2pt,outer sep=1pt},
  bigloopabove/.style={loop above,out=130,in=50,min distance=15mm,looseness=4},
  bigloopbelow/.style={loop below,out=230,in=310,min distance=15mm,looseness=4},
  bigloopright/.style={loop right,out=35,in=325,min distance=15mm,looseness=4},
  edgelabel/.style={midway,fill=white,inner sep=1.2pt,font=\footnotesize}
]

\node[vertex] (u1) at (0.25,5.75) {$u_1$};
\node[vertex] (u2) at (2.55,3.25) {$u_2$};
\node[vertex] (v)  at (-1.55,1.55) {$v$};
\node[vertex] (w)  at (7.00,3.25) {$w$};
\node[vertex] (y)  at (7.00,0.20) {$y$};
\node[vertex] (z1) at (0.75,0.75) {$z_1$};
\node[vertex] (z2) at (2.65,0.85) {$z_2$};
\node[vertex] (z3) at (2.70,-0.85) {$z_3$};
\node[vertex] (z4) at (0.75,-0.85) {$z_4$};

\draw[->] (u1) edge[bigloopabove]
  node[edgelabel] {$\alpha$} (u1);
\draw[->] (u1) edge[bigloopbelow]
  node[edgelabel] {$\beta$} (u1);

\draw[->] (u2) edge[bigloopabove]
  node[edgelabel] {$\gamma$} (u2);
\draw[->] (u2) edge[bigloopbelow]
  node[edgelabel] {$\delta$} (u2);

\draw[->] (v) edge[bigloopabove]
  node[edgelabel] {$c$} (v);

\draw[->] (y) edge[bigloopright]
  node[edgelabel] {$\varepsilon$} (y);

\draw[->] (u1) to[bend left=8]
  node[edgelabel,pos=0.55,sloped] {$f_1$} (w);
\draw[->] (u2) to[bend left=2]
  node[edgelabel,pos=0.53,sloped] {$f_2$} (w);
\draw[->] (z2) to[bend right=8] (w);

\draw[->] (v) to[bend right=5]
  node[edgelabel,pos=0.48,sloped] {$d_1$} (u2);
\draw[->] (v) to[bend left=4]
  node[edgelabel,pos=0.48,sloped] {$d_2$} (z1);
\draw[->] (v)
  .. controls (-0.95,-2.45) and (4.95,-2.45) ..
  node[edgelabel,pos=0.60,sloped] {$d_3$} (y);

\draw[->] (z1) to[bend left=22]
  node[edgelabel,pos=0.50,sloped] {$e_1$} (z2);
\draw[->] (z2) to[bend left=18]
  node[edgelabel,pos=0.50,sloped] {$e_2$} (z3);
\draw[->] (z3) to[bend left=22]
  node[edgelabel,pos=0.50,sloped] {$e_3$} (z4);
\draw[->] (z4) to[bend left=18]
  node[edgelabel,pos=0.50,sloped] {$e_4$} (z1);

\end{tikzpicture}

    \end{center}
Then ${}_HE$ is the graph
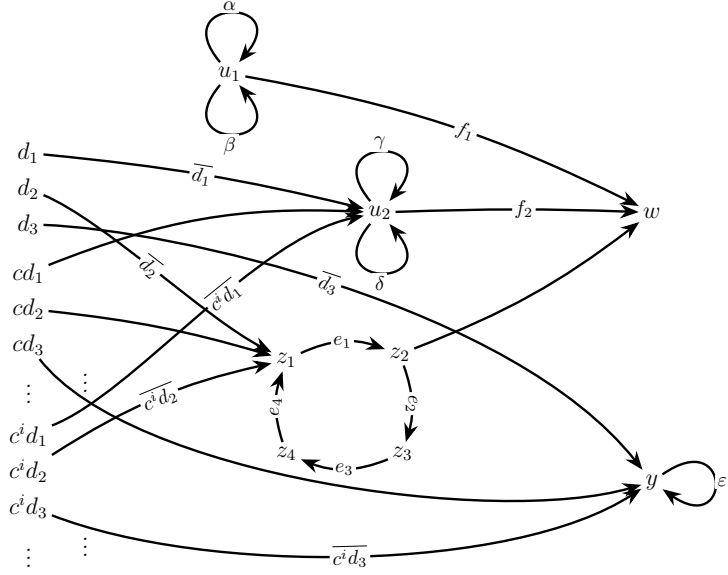
\begin{center}
    \begin{tikzpicture}[
    scale=0.75,
  transform shape,
  >=Stealth,
  line width=0.9pt,
  vertex/.style={inner sep=1.2pt,outer sep=1pt},
  bigloopabove/.style={loop above,out=130,in=50,min distance=15mm,looseness=4},
  bigloopbelow/.style={loop below,out=230,in=310,min distance=15mm,looseness=4},
  bigloopright/.style={loop right,out=35,in=325,min distance=15mm,looseness=4},
  edgelabel/.style={midway,fill=white,inner sep=1.2pt,font=\footnotesize}
]

\node[vertex] (u1) at (0.55,5.75) {$u_1$};
\node[vertex] (u2) at (3.20,3.30) {$u_2$};
\node[vertex] (w)  at (8.00,3.30) {$w$};
\node[vertex] (y)  at (8.00,-1.45) {$y$};

\node[vertex] (z1) at (1.55,0.70) {$z_1$};
\node[vertex] (z2) at (3.55,0.80) {$z_2$};
\node[vertex] (z3) at (3.60,-0.95) {$z_3$};
\node[vertex] (z4) at (1.55,-0.95) {$z_4$};

\draw[->] (u1) edge[bigloopabove]
  node[edgelabel] {$\alpha$} (u1);
\draw[->] (u1) edge[bigloopbelow]
  node[edgelabel] {$\beta$} (u1);

\draw[->] (u2) edge[bigloopabove]
  node[edgelabel] {$\gamma$} (u2);
\draw[->] (u2) edge[bigloopbelow]
  node[edgelabel] {$\delta$} (u2);

\draw[->] (y) edge[bigloopright]
  node[edgelabel] {$\varepsilon$} (y);

\draw[->] (u1) to[bend left=8]
  node[edgelabel,pos=0.55,sloped] {$f_1$} (w);
\draw[->] (u2) to[bend left=2]
  node[edgelabel,pos=0.53,sloped] {$f_2$} (w);
\draw[->] (z2) to[bend right=8] (w);

\draw[->] (z1) to[bend left=22]
  node[edgelabel,pos=0.50,sloped] {$e_1$} (z2);
\draw[->] (z2) to[bend left=18]
  node[edgelabel,pos=0.50,sloped] {$e_2$} (z3);
\draw[->] (z3) to[bend left=22]
  node[edgelabel,pos=0.50,sloped] {$e_3$} (z4);
\draw[->] (z4) to[bend left=18]
  node[edgelabel,pos=0.50,sloped] {$e_4$} (z1);


\node[vertex] (d1) at (-3.00,4.35) {$d_1$};
\node[vertex] (d2) at (-3.00,3.72) {$d_2$};
\node[vertex] (d3) at (-3.00,3.09) {$d_3$};

\node[vertex] (cd1) at (-3.00,2.23) {$c d_1$};
\node[vertex] (cd2) at (-3.00,1.60) {$c d_2$};
\node[vertex] (cd3) at (-3.00,0.97) {$c d_3$};

\node at (-3.00,0.20) {$\vdots$};

\node at (-2.00,0.40) {$\vdots$};

\node[vertex] (cid1) at (-3.00,-0.62) {$c^i d_1$};
\node[vertex] (cid2) at (-3.00,-1.25) {$c^i d_2$};
\node[vertex] (cid3) at (-3.00,-1.88) {$c^i d_3$};

\node at (-3.00,-2.68) {$\vdots$};
\node at (-2.00,-2.50) {$\vdots$};
\draw[->] (d1) to[bend left=4]
  node[edgelabel,pos=0.49,sloped,inner xsep=1.5pt] {$\overline{d_1}$} (u2);
\draw[->] (cd1) to[bend left=12] (u2);
\draw[->] (cid1)
  .. controls (-1.15,0.10) and (1.00,2.85) ..
  node[edgelabel,pos=0.58,sloped,inner xsep=1.5pt] {$\overline{c^i d_1}$} (u2);

\draw[->] (d2)
  .. controls (-1.65,3.15) and (0.10,1.45) ..
  node[edgelabel,pos=0.48,sloped,inner xsep=1.5pt] {$\overline{d_2}$} (z1);
\draw[->] (cd2) to[bend left=5] (z1);
\draw[->] (cid2) to[bend left=10]
  node[edgelabel,pos=0.51,sloped,inner xsep=1.5pt] {$\overline{c^i d_2}$} (z1);

\draw[->] (d3)
  .. controls (-0.45,3.00) and (5.55,1.65) ..
  node[edgelabel,pos=0.48,sloped,inner xsep=1.5pt] {$\overline{d_3}$} (y);
\draw[->] (cd3)
  .. controls (-1.20,-1.35) and (5.25,-2.25) .. (y);
\draw[->] (cid3)
  .. controls (0.15,-3.10) and (5.55,-3.05) ..
  node[edgelabel,pos=0.49,sloped,inner xsep=1.5pt] {$\overline{c^i d_3}$} (y);

\end{tikzpicture}
\end{center}
while $E/H$ is the graph
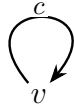
\begin{center}
 \begin{tikzpicture}[
  >=Stealth,
  line width=0.9pt,
  vertex/.style={inner sep=1.2pt,outer sep=1pt},
  bigloopabove/.style={loop above,out=130,in=50,min distance=15mm,looseness=4},
  bigloopbelow/.style={loop below,out=230,in=310,min distance=15mm,looseness=4},
  bigloopright/.style={loop right,out=35,in=325,min distance=15mm,looseness=4},
  edgelabel/.style={midway,fill=white,inner sep=1.2pt,font=\footnotesize}
]


\node[vertex] (v)  at (-1.55,1.55) {$v$};
\draw[->] (v) edge[bigloopabove]
  node[edgelabel] {$c$} (v);
\end{tikzpicture}    
\end{center}
\end{example}

Before describing the injective envelope of the simple left $L_K(E)$-module $L_K(E)w$, we establish the following property of the paths in $F_E(H)$.
\begin{lemma}\label{lemma:mu*lambda}
For each 
$\mu, \lambda\in F_E(H)=\{c^id_j\mid i\geq 0, 1\leq j\leq n\}$ we have
\[\mu^*\lambda=\begin{cases}
    r(\mu)  & \text{if }\mu=\lambda, \\
      0& \text{otherwise}.
\end{cases}\]
\end{lemma}
\begin{proof}
Let $\mu=c^{i_1}d_{j_1}$ and $\lambda=c^{i_2}d_{j_2}$.
If
\[
0\not=\mu^*\lambda=d_{j_1}^*(c^*)^{i_1}c^{i_2}d_{j_2}
\]
one gets easily $i_1=i_2$, and $j_1=j_2$.
\end{proof}

We now recall that, by \cite[Corollary 4.6]{AMT26}, since the set $E^0$ of vertices in $E$ is finite, the injective envelope of the simple left $L_K(E)$-module $L_K(E)w$ is the module
\begin{align*}
 X^E_w&=   \{\sum_{\lambda\in \Path({}_HE)w}k_\lambda\lambda\mid k_\lambda\in K, |\{s(\lambda);k_\lambda\not=0\}|\text{ is finite}\}\\
 &=\{\sum_{\lambda\in\Path(E)w}k_\lambda \lambda\mid k_\lambda\in K\}
\end{align*}
of formal series of paths in $E$ ending with $w$.
Each element $\xi$ of $X^E_w$ can be written as
\[
\xi=v\xi+(\xi-v\xi).\]
Consider the subset
\begin{align*}
 vX^E_w:=\{v\cdot\xi\mid\xi\in X^E_w\}&=\{\sum_{\lambda\in \Path(E)w}k_\lambda \lambda\mid s(\lambda)=v, k_\lambda\in K\}\\
 &=\{\sum_{\lambda\in v\Path(E)w}k_\lambda \lambda\mid k_\lambda\in K\}
\end{align*}
The paths belonging to $v\Path(E) w$ have the form
\[c^id_j\nu\quad i\geq 0, j\in\{1,..,n\}\]
where the $c$-\emph{tail} $d_j\nu$ is a path in $v\Path(E)w$ not starting with a traverse of the loop $c$.
Denote by $Y_{v,w}$ the set of all possible $c$-tails of the paths in $v\Path(E) w$. Collecting the paths with the same $c$-tail, for any $\xi\in X^E_w$ we have
\[v\cdot\xi=\sum_{\mu\in Y_{v,w}}A_\mu(c)\mu\]
where $A_\mu(c)$ is the evaluation in $c$ of a formal series $A_\mu(x)$ belonging to the $K$-algebra $K[[x]]$.

Consider the following two submodules of $X^E_w$:
\begin{align*}
    \overline{L_K(E)w}:&=\{\xi\in X^E_w\mid v\xi=\sum_{\mu\in Y_{v,w}}p_\mu(c)\mu,\\
    &p_\mu(x)\in K[x], \exists N_\xi\in\mathbb N: \deg p_\mu(x)\leq N_\xi\}
\end{align*}
\begin{align*}
    D(L_K(E)w):=\{\xi\in X^E_w\mid v\xi=&\sum_{\mu\in Y_{v,w}}A_\mu(c)\mu,\\
    &A_\mu(x)\in K[x]_{\mathfrak M}\}.
\end{align*}
It is $L_K(E)w\subseteq \overline{L_K(E)w}\subsetneq D(L_K(E)w)\subsetneq X^E_w$.
Since $D(L_K(E)w)$ contains the simple $L_K(E)w$ and is properly contained in $X^E_w$, it is not an injective left $L_K(E)$-module. We will prove that it is a divisible left $L_K(E)$-module.

To this end, we first establish some preliminary results. We start by describing the images of homomorphisms defined on left ideals contained in $I(H)$.

\begin{lemma}\label{lemma:image}
    Let $J$ be a left ideal contained in the two-sided ideal $I(H)$ generated by $H:=E^0\setminus\{v\}$, and let $\varphi:J\to X^E_w$ be a homomorphism of the left $L_K(E)$-modules. Then 
    \[\Imm \varphi\subseteq \overline{L_K(E)w}.\]
\end{lemma}
\begin{proof}
Let $j\in J$ and $v\varphi(j)=\varphi(vj)=\displaystyle\sum_{\mu\in Y_{v,w}}A_\mu(c)\mu$ with $A_\mu(x)\in K[[x]]$. Let us prove that $A_\mu(x)\in K[x]$ for each $\mu\in Y_{v,w}$, and that the degrees of these polynomials are bounded by a suitable $N_j\in\mathbb N$. Fix $\overline{\mu}\in Y_{v,w}$.
By \cite[Lemma 2.4.1]{AAM}, 
    \[vj=\sum_{i=1}^m k_i\gamma_i\delta^*_i\qquad s(\gamma_i)=v, r(\gamma_i)=r(\delta_i)\in H, i=1,...,m.\]
    For $i=1,...,m$, we have
\[\gamma_i=c^{\ell_i}d_{j_i}\gamma'_i,\quad \ell_i\geq 0, 1\leq j_i\leq n, s(\gamma_i')\in H.\]
Thus, if $N_j>\max\{\ell_i\mid i=1,...,m\}$, we have
\begin{align*}
    (c^*)^{N_j} j&=(c^*)^{N_j} vj=(c^*)^{N_j}\sum_{i=1}^m k_i\gamma_i\delta^*_i\\
    &=(c^*)^{N_j}\sum_{i=1}^m k_i c^{\ell_i}d_{j_i}\gamma'_i\delta^*_i=0.
\end{align*}
Then
\begin{align*}
  0=\varphi((c^*)^{N_j}j)=(c^*)^{N_j}\varphi(vj)=(c^*)^{N_j}\sum_{\mu\in Y_{v,w}}A_\mu(c)\mu
\end{align*}
Hence $A_\mu(x)$, $\mu\in Y_{v,w}$, are polynomials of degree $<N_j$, and $\varphi(j)\in \overline{L_K(E)w}$.
\end{proof}

We now verify the conditions of the divisibility criterion established in the previous section (see \Cref{Steps} with $I=I(H)$). We start with the following result concerning left ideals contained in $I(H)$.

\begin{proposition}[Step 1 of \Cref{Steps}]\label{prop:Step1}
    If $J$ is a left $L_K(E)$-ideal contained in $I(H)$, then each homomorphism $\varphi:J\to D(L_K(E)w)$ extends to a homomorphism $I(H)\to D(L_K(E)w)$.
\end{proposition}
\begin{proof}
As observed in \Cref{hedgerem}, the two sided ideal $I(H)$ is isomorphic as $K$-algebra to $L_K({}_HE)$. Therefore the categories of unitary left $I(H)$-modules and that of left $L_K({}_HE)$-modules can be identified.
Claim 1: The left $L_K(E)$-module $\overline{L_K(E)w}$ is also a left $I(H)$-module.\\
The left module action of $I(H)$ on $\overline{L_K(E)w}$ is induced by $I(H)\subseteq L_K(E)$. Let us check that $\overline{L_K(E)w}$ is a unitary left $I(H)$-module, i.e.,
\[I(H)\overline{L_K(E)w}=\overline{L_K(E)w}.\]
    Let $\xi
\in \overline{L_K(E)w}$. We have
    \[v\xi=\sum_{\mu\in Y_{v,w}}p_{\mu}(c)d_{i_\mu}\mu'\]
where $d_{i_\mu}$, $1\leq i_\mu\leq n$, are exits of $c$, $s(\mu')\in H$, and  $p_\mu(x)\in K[x]$ with $\deg p_\mu(x)\leq N_\xi$ for a suitable integer $N_\xi$ depending only on $\xi$. 
    Consider the following set of idempotent elements in $I(H)$:
    \[S_\xi=H\cup \{(c^t d_i)(c^t d_i)^*\mid 0\leq t\leq N_\xi, d_i\text{ exit of }c\text{ for }1\leq i\leq n\}.\]
    The set $S_\xi$ is finite: indeed $H$ is finite, $t\leq N_\xi$, and the number of exits of $c$ is finite. Moreover, its elements are pairwise orthogonal. 
  It is easy to check that
    \[(\sum_{u\in S_\xi}u)\xi=(\sum_{u\in S_\xi\setminus H}u)v\xi+(\sum_{u\in H}u)(\xi-v\xi)=v\xi+(\xi-v\xi)=\xi.\]
    By the arbitrary choice of $\xi$, it follows that $\overline{L_K(E)w} $ is a unitary left $I(H)$-module.\\
Claim 2: The left $I(H)$-module $\overline{L_K(E)w} $ is isomorphic to the injective left $I(H)$-module $X^{{}_HE}_w$.\\    
By \cite[Corollary 4.6]{AMT26}, the injective envelope of the left $L_K({}_HE)$-module $L_K({}_HE)w$ is
\[
X^{{}_HE}_w=\{\sum_{\lambda\in \Path({}_HE)w}k_\lambda\lambda\mid k_\lambda\in K, |\{s(\lambda);k_\lambda\not=0\}|\text{ is finite}.\}
\]
As observed at the beginning of the proof, $X^{{}_HE}_w$ is an injective left $I(H)$-module.
For each element $\xi\in X^{{}_HE}_w$, for any exit $d_i$, $i=1,..., n$, of $c$ in $E$ it is
\[
d_i^*(c^*)^t\cdot\xi\equiv (\overline{c^td_i})^*\xi=0\qquad t\gg 0.
\]
Then, the map between $\Path({}_HE)w$ and $\Path(E)w$
\[
\lambda\mapsto \begin{cases}
   \lambda   & \text{if }s(\lambda)\in H, \\
    c^td_i\lambda'  & \text{if }\lambda=\overline{c^td_i}\lambda'
\end{cases}
\]
induces an isomorphism of left $I(H)$-modules between $X^{{}_HE}_w$ and $\overline{L_K(E)w}$.
%
Claim 3: Conclusion of the proof.\\
By Claim 2, $\overline{L_K(E)w}$ is an injective left $I(H)$-module. Since $D(L_K(E)w)$ is contained in $X^E_w$, by \Cref{lemma:image} $\Imm\varphi\subseteq \overline{L_K(E)w}$.
    Thus we have the following commutative diagram
    \[
\xymatrix{J\ar@{-->}[ddd]_{\varphi}\ar[dr]^{\overline\varphi}\ar@{^(->}[rrr]^{\iota_1}&&&I(H)\ar@{.>}[lld]_\psi\\
    &\overline{L_K(E)w}\ar@{^(-->}[ddl]^{\iota_2}\\
    {}\\
    D(L_K(E)w)
    }
    \]
where the dashed arrows are homomorphisms of left $L_K(E)$-modules, the solid arrows are homomorphisms of left $L_K(E)$- and $I(H)$- modules, and the dotted arrow $\psi:I(H)\to \overline{L_K(E)w}$ is a homorphism of left $I(H)$-modules such that $\psi\circ \iota_1=\overline\varphi$. If $\psi$ were a homomorphism of left $L_K(E)$-modules, we would get $\iota_2\circ \psi\circ \iota_1=\varphi$ thus arriving at the thesis. Let us prove it.
    Given $i\in I(H)$ and $r\in L_K(E)$, we have to check that $\psi(ri)=r\psi(i)$. Since $I(H)$ has local units, there exists $i_0\in I(H)$ such that $i_0i=i$. Therefore
    \[
    \psi(ri)=\psi(r(i_0i))=\psi((ri_0)i)=(ri_0)\psi(i)=r(i_0\psi(i))=r\psi(i_0i)=r\psi(i).
    \]   
\end{proof}

We next consider principal left ideals contained in $I(H)$.
\begin{proposition}[Step 2 of \Cref{Steps}]\label{prop:Step2}
    For each $r\in I(H)$, any left $L_K(E)$-homomorphism $\varphi:L_K(E)r\to D(L_K(E)w)$ lifts to a left $L_K(E)$-homomorphism $\hat{\varphi}:L_K(E)\to D(L_K(E)w)$.
\end{proposition}
\begin{proof}
By \cite[Proposition~3.4]{AMT26}, the two sided ideal $I(H)$ is equal to the following direct sum of left ideals:
\[(\oplus_{u\in H}L_K(E)u)\bigoplus(\oplus_{\mu\in F_E(H)}L_K(E)\mu^*).\]
If $r\in I(H)$, then it belongs to
\[I_F:=(\oplus_{u\in H}L_K(E)u)\bigoplus(\oplus_{\mu\in F}L_K(E)\mu^*)\]
for a suitable finite subsets $F$ of $F_E(H)$.
By \Cref{prop:Step1}, $\varphi$ extends to an homomorphism of left $R$-modules $\overline\varphi:I(H)\to D(L_K(E)w)$. We will conclude, lifting the restriction $\psi:=\overline\varphi_{\mid I_F}:I_F\to D(L_K(E)w)$ of $\overline\varphi$ to the $K$-algebra $L_K(E)$. By \Cref{lemma:image}, $\Imm\psi\subseteq \overline{L_K(E)w}$.
For any $u\in H$, we have
\[\psi(u)=u\psi(u)=u(\sum_{\lambda\in \Path(E)w}k_{\lambda,u} \lambda)=\sum_{\lambda\in u\Path(E)w}k_{\lambda,u} \lambda.\]
If $\eta\in F\subseteq F_E(H)$, then $r(\eta)\in H$: indeed $\eta= c^{j_\eta} d_{i_\eta}$ for suitable $j_\eta\geq 0$, and exit $d_{i_\eta}$, $1\leq i_\eta\leq n$, of $c$. Then
\begin{align*}
    \psi(\eta^*)&=r(\eta)\psi(\eta^*)=r(\eta)(\sum_{\lambda\in \Path(E)w}k_{\lambda,\eta} \lambda)\\
    &=\sum_{\lambda\in r(\eta)\Path(E)w}k_{\lambda,\eta} \lambda
\end{align*}
We now define the $\hat\varphi:R\to D(L_K(E)w)$ setting
\[
\hat\varphi(1):=\sum_{u\in H}\sum_{\lambda\in u\Path(E)w}k_{\lambda,u} \lambda+\sum_{\mu\in F}\mu\sum_{\lambda\in r(\mu)\Path(E)w}k_{\lambda,\mu} \lambda.
\]
Observe that since $F$ is finite, $\max\{j_\eta\mid \eta=c^{j_\eta} d_{i_\eta}\in F\}$ is a natural number and hence $\hat\varphi(1)$ belongs to $\overline{L_K(E)w}\subseteq D(L_K(E)w)$.
Let us check that $\hat\varphi$ lifts $\varphi$:
\begin{align*}
\hat\varphi(u)&=u\hat\varphi(1)=u\cdot\left(\sum_{u\in H}\sum_{\lambda\in u\Path(E)w}k_{\lambda,u} \lambda+\sum_{\mu\in F}\mu\sum_{\lambda\in r(\mu)\Path(E)w}k_{\lambda,\mu} \lambda
    \right)\\
    &=\sum_{\lambda\in u\Path(E)w}k_{\lambda,u} \lambda=\psi(u)\qquad\forall u\in H.
\end{align*}
\begin{align*}
    \hat\varphi(\eta^*)&=\eta^*\hat\varphi(1)=\eta^*\cdot \left(\sum_{u\in H}\sum_{\lambda\in u\Path(E)w}k_{\lambda,u} \lambda+\sum_{\mu\in F}\mu\sum_{\lambda\in r(\mu)\Path(E)w}k_{\lambda,\mu} \lambda
    \right)\\
    &\text{by \Cref{lemma:mu*lambda}}\\
    &=\sum_{\lambda\in r(\eta)\Path(E)w}k_{\lambda,\eta} \lambda
    =\psi(\eta^*).
\end{align*}
\end{proof}

It remains to consider the last case of left ideals generated by $I(H)$ and a fixed element outside $I(H)$.
\begin{proposition}[Step 3 of \Cref{Steps}]
Let $r\in L_K(E)\setminus I(H)$. Any homomorphism of left $L_K(E)$-modules $\varphi:L_K(E)r+I(H)\to D(L_K(E)w)$ lifts to an homomorphism $L_K(E)\to D(L_K(E)w)$.
\end{proposition}
\begin{proof}
Assume $L_K(E)r+I(H)\not=L_K(E)$.
    The quotient $L_K(E)/I(H)$ is isomorphic to $K[x,x^{-1}]$, where $x$ corresponds to $c+I(H)$. Since $I(H)$ is properly contained in $L_K(E)r+I(H)$, the quotient $(L_K(E)r+I(H))/I(H)$ is a non-zero proper ideal $\langle p(x)\rangle$
    of $K[x,x^{-1}]$ with $p(x)=p_\ell x^\ell+\cdots+p_1x+1$, $p_\ell\not=0$. Therefore
    \[L_K(E)r+I(H)=L_K(E)(p_\ell c^\ell +\cdots+p_1c+v)+I(H).\]
    Consider 
    \[
    z:=p_\ell c^\ell +\cdots+p_1c+v+\sum_{u\in H}u=p_\ell c^\ell +\cdots+p_1c+1_{L_K(E)}.
    \]
    It is clear that $L_K(E)z$ is contained in $L_K(E)(p_\ell c^\ell +\cdots+p_1c+v)+I(H)$. 
Let us prove $L_K(E)z=L_K(E)(p_\ell c^\ell +\cdots+p_1c+v)+I(H)$. Indeed, first $p_\ell c^\ell +\cdots+p_1c+v=v\cdot z\in L_K(E)z$. Second by \cite[Proposition 3.4]{AMT26} it is
\[I(H)=\bigoplus_{\mu\in H\cup F_E(H)}L_K(E)\mu^*.\]
For each $u\in H$, one has $u=uz\in L_K(E)z$. If $\mu\in F_E(H)$, then $\mu=c^jd_i$ with $j\geq 0$, and $d_i$, $1\leq i\leq n$, exit of $c$. Let us check by induction on $j\geq 0$ that $d_i^*(c^*)^j\in L_K(E)z$. For $j=0$
\begin{align*}
    d_i^*&=d_i^*(p_\ell c^\ell +\cdots+p_1c+1_{L_K(E)})=d_i^*z\in L_K(E)z
\end{align*}
For $0<j\leq \ell $,
\begin{align*}
    d_i^*(c^*)^jz&=d_i^*(c^*)^j(p_\ell c^\ell +\cdots+p_1c+1_{L_K(E)})\\
    &=p_jd_i^*+p_{j-1}d_i^*c^*+\cdots+p_1 d_i^*(c^*)^{j-1}+d_i^*(c^*)^j.
\end{align*}
Then by inductive hypothesis 
\[
d_i^*(c^*)^j=d_i^*(c^*)^jz
-(p_jd_i^*+p_{j-1}d_i^*c^*+\cdots+p_1 d_i^*(c^*)^{j-1})\in L_K(E)z.
\]
Finally, for $j>\ell $
\begin{align*}
    d_i^*(c^*)^jz&=d_i^*(c^*)^j(p_\ell c^\ell +\cdots+p_1c+1_{L_K(E)})\\
    &=p_\ell d_i^*(c^*)^{j-\ell }+\cdots+p_1d_i^*(c^*)^{j-1}+d_i^*(c^*)^j.
\end{align*}
Then by inductive hypothesis 
\[
d_i^*(c^*)^j=d_i^*(c^*)^jz-(p_\ell d_i^*(c^*)^{j-\ell }+\cdots+p_1d_i^*(c^*)^{j-1})\in L_K(E)z.
\]
Let $A(x)\in K[x]_{\mathfrak M}$ the inverse of $p(x)$.
Then setting
\[\hat\varphi(1)=(A(c)+\sum_{u\in H}u)\varphi(z)\in D(L_K(E)w)\]
we define an homomorphism of left $L_K(E)$-modules $\hat\varphi:L_K(E)\to D(L_K(E)w)$ extending $\varphi$.
\end{proof}

Combining the divisibility criterion \Cref{Steps} with the above results, we conclude that $D(L_K(E)w)$ is a divisible left $L_K(E)$-module. As observed above, $D(L_K(E)w)$ is not injective, and hence this construction yields an example of a divisible non-injective left $L_K(E)$-module.

\end{document}